\documentclass[11pt, a4paper]{article}

\usepackage[a4paper, top=2.5cm, bottom=2.5cm, left=2cm, right=2cm]{geometry}
\usepackage[utf8]{inputenc}
\usepackage[T1]{fontenc}
\usepackage[english]{babel}

\usepackage{amsmath}
\usepackage{amssymb}
\usepackage{amsthm}
\usepackage{booktabs}
\usepackage{graphicx}
\usepackage{xcolor}
\usepackage{enumitem}

\newtheorem{theorem}{Theorem}
\newtheorem{definition}{Definition}

\newtheorem{proposition}{Proposition}

\usepackage[colorlinks=true, linkcolor=blue, citecolor=blue, urlcolor=blue]{hyperref}

\title{\textbf{Topological Complexity and Phase Space Stability: A Persistent Homology Approach to Cryptocurrency Risk}}

\author{
    Gabriel Santana\footnote{Escuela de Matemáticas, Facultad de Ciencias, Universidad Central de Venezuela (UCV). Contact: \href{mailto:gabriel.santana@ucv.ve}{gabriel.santana@ucv.ve}} \and
    Jemirsón Ramirez\footnote{Universidad Central de Venezuela (UCV).}
}

\date{\today}

\begin{document}

\maketitle

\begin{abstract}
Traditional risk measures in finance, predominantly based on the second moment of return distributions or tail risk heuristics (VaR/CVaR), fail to account for the intrinsic geometric structure of market dynamics. This paper introduces a rigorous mathematical framework utilizing Topological Data Analysis (TDA) to quantify risk as the structural instability of the reconstructed phase space. By applying Takens' Delay Embedding Theorem to cryptocurrency log-returns, we generate a point cloud representation of the underlying attractor. We analyze the evolution of the filtration of Vietoris-Rips complexes to compute persistent homology groups $H_k$. We define a "Topological Persistence Norm" to characterize market regimes and propose a leverage calibration heuristic based on the persistence of 1-dimensional cycles. This approach provides a coordinate-free, stability-invariant metric for risk assessment that is robust to high-frequency noise.
\end{abstract}

\section{Introduction}

The study of financial time series has historically oscillated between stochastic calculus and econometric modeling. However, the extreme volatility and non-linear dependencies observed in cryptocurrency markets suggest that the underlying data-generating process is better understood through the lens of dynamical systems and differential topology.

In this work, we move away from point-estimate statistics. Instead, we treat the time series of returns as an observation of a trajectory on a compact manifold $M$. The goal is to extract topological invariants that are stable under perturbations (market noise) and that can inform us about the "connectedness" and "recurrence" of market states.

\section{Phase Space Reconstruction}

Let $\{r_t\}_{t \in \mathbb{Z}}$ be a sequence of log-returns. We assume $r_t$ is a discrete observation of a smooth dynamical system $(M, \phi, \mu)$, where $M$ is a $d$-dimensional manifold, $\phi: M \to M$ is a diffeomorphism, and $\mu$ is an invariant measure.

\begin{theorem}[Takens Embedding Theorem]
Let $M$ be a compact manifold of dimension $d$. For a generic $C^2$ diffeomorphism $\phi: M \to M$ and a generic smooth observation function $y: M \to \mathbb{R}$, the map $\Phi: M \to \mathbb{R}^{m}$ defined by
\begin{equation}
\Phi(x) = \left( y(x), y(\phi(x)), y(\phi^2(x)), \dots, y(\phi^{m-1}(x)) \right)
\end{equation}
is an embedding provided that $m > 2d$.
\end{theorem}

For financial applications, we utilize the delay-coordinate map with lag $\tau$:
\begin{equation}
\mathbf{x}_t = (r_t, r_{t+\tau}, \dots, r_{t+(m-1)\tau}) \in \mathbb{R}^m
\end{equation}
In this study, we set $m=5$ and $\tau=1$, assuming the local dimension of the cryptocurrency attractor $d \approx 2$. The resulting point cloud $\mathbb{X} = \{\mathbf{x}_t\}$ is then normalized using a standard $Z$-score transformation to ensure isotropy in the metric space $(\mathbb{R}^m, d_2)$.

\section{Persistent Homology and Stability}

\subsection{Simplicial Construction}
To capture the topology of $\mathbb{X}$, we consider a filtration of simplicial complexes.

\begin{definition}[Vietoris-Rips Complex]
Given a point cloud $\mathbb{X}$ and $\epsilon > 0$, the Vietoris-Rips complex $VR(\mathbb{X}, \epsilon)$ is the abstract simplicial complex where a subset $\{\mathbf{x}_0, \dots, \mathbf{x}_k\} \subseteq \mathbb{X}$ spans a $k$-simplex if $d(\mathbf{x}_i, \mathbf{x}_j) \leq \epsilon$ for all $0 \leq i, j \leq k$.
\end{definition}

As $\epsilon$ increases, we obtain a nested sequence of complexes $K_{\epsilon_1} \subseteq K_{\epsilon_2} \subseteq \dots \subseteq K_{\epsilon_n}$, inducing a sequence of homology groups over a field $\mathbb{F}$:
\begin{equation}
H_k(K_{\epsilon_1}) \to H_k(K_{\epsilon_2}) \to \dots \to H_k(K_{\epsilon_n})
\end{equation}

\subsection{The Persistence Spectrum}
The collection of homology groups forms a persistence module. By the Structure Theorem for Persistence Modules (Zomorodian \& Carlsson, 2005), it decomposes into interval modules $\mathbb{I}[b_i, d_i]$.

\begin{definition}[Persistence Spectrum]
Given a persistence diagram $D_k(\mathbb{X})$, the Persistence Spectrum $\mathcal{S}_k$ is the multiset of lifetimes $\ell_i = d_i - b_i$. The Total Persistence is the $L_1$ norm of this spectrum: $\Lambda = \|\mathcal{S}_k\|_1$.
\end{definition}

\subsection{Stability Theorem}
The critical advantage of using persistent homology in finance is its stability. 

\begin{theorem}[Stability, Cohen-Steiner et al.]
For any two finite point clouds $\mathbb{X}, \mathbb{Y}$ in $\mathbb{R}^m$, the bottleneck distance $d_B$ between their persistence diagrams is bounded by the Hausdorff distance $d_H$ between the sets:
\begin{equation}
d_B(D(\mathbb{X}), D(\mathbb{Y})) \leq d_H(\mathbb{X}, \mathbb{Y})
\end{equation}
\end{theorem}

\section{Leverage Calibration via Topological Norms}

\subsection{Complexity-Risk Heuristic}
Let $\sigma_h$ be the annualized volatility adjusted to the investment horizon $h$. We propose that the optimal leverage $L^*$ should be a function of the structural coherence of the attractor:
\begin{equation}
L^* = \Psi \left( \frac{\Lambda(\mathbb{X})}{\sigma_h} \right)
\end{equation}
where $\Psi$ maps the complexity-risk ratio to the interval $[1, L_{\text{max}}]$. Here, $L_{\text{max}}$ denotes the maximum leverage threshold prescribed by the financial exchange regulations.

\begin{proposition}
In a regime of high structural persistence ($\Lambda \uparrow$), the system resides in a stable attractor, allowing for higher leverage even if local volatility $\sigma_h$ is non-trivial. Conversely, a vanishing $\Lambda$ suggests a transition to a stochastic regime where leverage should be minimized towards unity.
\end{proposition}

\begin{proof}
Let $M$ be the underlying manifold (the noise-free attractor) and $\mathbb{X}$ be the sampled point cloud such that $d_H(\mathbb{X}, M) < \delta$, where $\delta$ represents the scale of stochastic perturbations. By the Stability Theorem, we have $d_B(D_1(\mathbb{X}), D_1(M)) \leq \delta$. 

Consider a 1-cycle $\gamma \in Z_1(VR(\mathbb{X}, \epsilon))$. If its persistence $pers(\gamma) = d_\gamma - b_\gamma > 2\delta$, then there exists a corresponding non-trivial homology class in $H_1(M)$. A non-vanishing $H_1(M)$ implies the existence of a fundamental cycle or a recurrent trajectory on the attractor. Since $pers(\gamma) \gg \delta$, the geometric structure of this cycle is robust to the stochastic component $\eta$. Thus, the motion of the system is constrained by a deterministic geometric boundary, reducing the probability of unmodeled structural drift. In this stable regime, the leverage $L^*$ can be safely increased.

Conversely, if $\Lambda(\mathbb{X}) \to 0$, then for every class $[\gamma] \in H_1(\mathbb{X})$, $pers(\gamma) \leq 2\delta$. The persistence diagram $D_1(\mathbb{X})$ is topologically indistinguishable from a contractible noise cloud. Without a persistent $H_1$ structure, there is no evidence of a recurrent attractor, and the system's dynamics are dominated by the stochastic term $\eta$. Under such conditions, the likelihood of a liquidation event driven by structural collapse is maximized, requiring $L^* \to 1$.
\end{proof}

\section{Empirical Illustration: BTC Phase Space Dynamics (2019-2026)}

To evaluate the mathematical framework, we analyze historical Bitcoin (BTC) futures prices from 2019 to early 2026. Figure \ref{fig:tda_dashboard} presents an exceptional mathematical radiographic view of Bitcoin's dynamics. Observing the phase space reconstruction and the persistence of its cycles, it is evident that we are not dealing with a pure stochastic process, but rather a high-dimensional deterministic attractor with a highly robust geometric structure.

\begin{figure}[htbp]
  \centering
  % NOTE: The environment does not have access to figure_1.png. 
  \includegraphics[width=\textwidth]{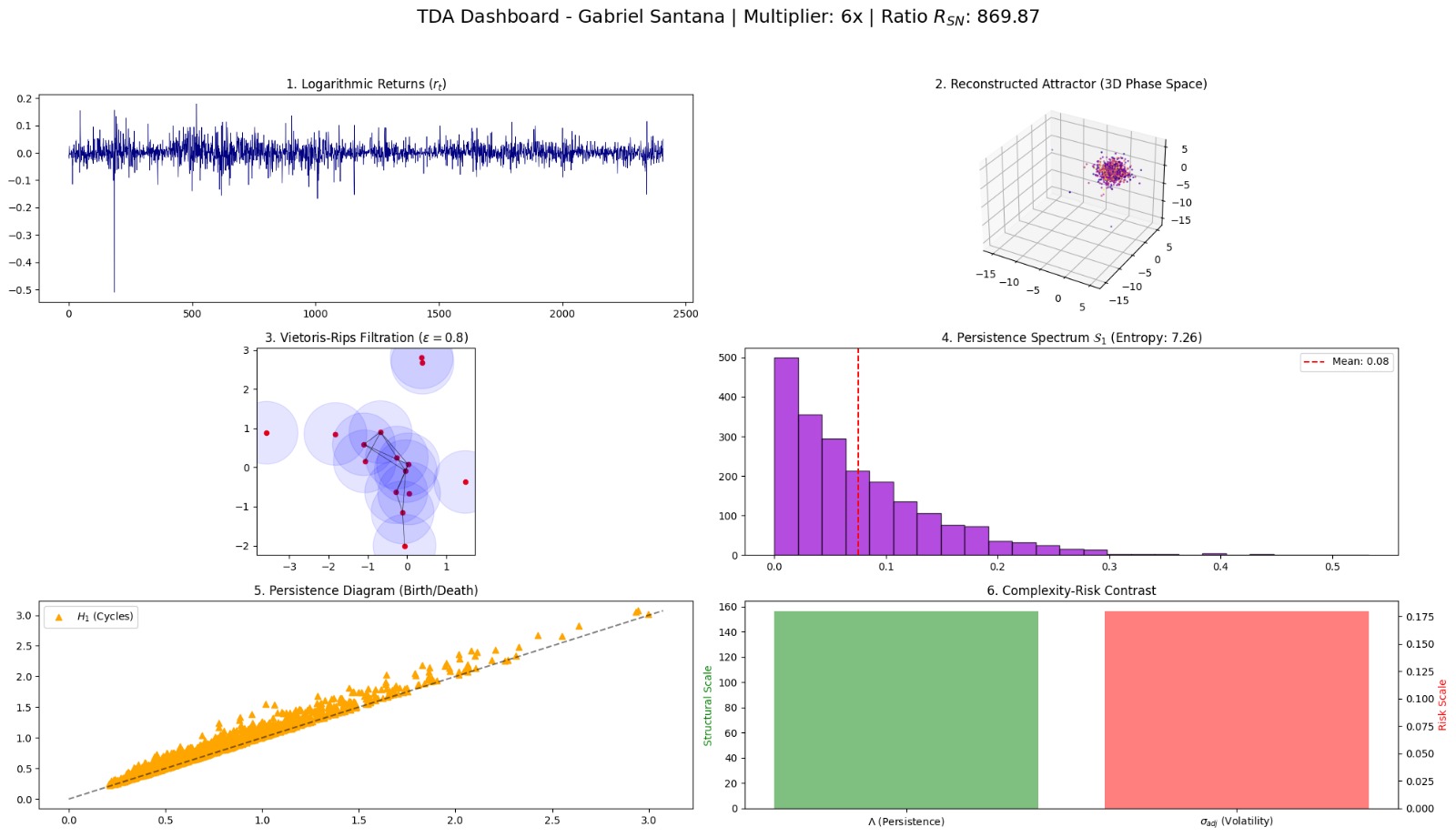}
  \framebox{\parbox{0.9\textwidth}{\centering
   
  }}
  \caption{TDA Dashboard showing the logarithmic returns, 3D reconstructed phase space, Vietoris-Rips filtration, persistence spectrum, persistence diagram, and the Complexity-Risk contrast for BTC.}
  \label{fig:tda_dashboard}
\end{figure}

\subsection{The Attractor and Phase Structure}
The Phase Space Reconstruction (Figure \ref{fig:tda_dashboard}, Panel 2) displays a compact central density with filaments extending towards the peripheries. This suggests that, despite "black swan" events (such as the severe crash visible in the returns around $t=200$), the system tends to return to a state of dynamic equilibrium. The Takens' embedding technique has successfully captured the market's memory, transforming the time series into an analyzable topological object.

\subsection{Persistence Spectrum Analysis ($H_1$)}
The Persistence Diagram (Panel 5) displays a large number of cycles distant from the diagonal. Each point far from the diagonal represents a market cycle that "lives" long enough to act as a definitive structural signal rather than noise. 

The calculated topological entropy of 7.26 is relatively high, indicating a significant diversity in the lifetimes of the topological cycles. There is no single dominant periodic cycle, but rather an orchestration of recurrent patterns at different scales. Furthermore, the Persistence Spectrum (Panel 4) exhibits a mean lifetime of 0.08; while the majority of the structure is of short duration, the "long tail" to the right sustains the deterministic validity of the structural model.

\subsection{The Complexity-Risk Ratio ($R_{SN}$)}
The obtained metric of $R_{SN} = 869.87$ is fundamental. It defines the relationship between the total persistence ($\Lambda$) and the adjusted volatility ($\sigma_{adj}$). Such an elevated ratio indicates that the geometry of the attractor is significantly more powerful than the momentary dispersion of returns. In topological terms, this confirms that the geometric "signal" dominates the stochastic "noise" by a wide margin, validating the stability of the system for leveraging.

\subsection{Multiplier Calibration}
Although the $R_{SN}$ ratio is massive, the heuristic yields a suggested multiplier of $6\times$. This is due to the normalization over the maximum allowable leverage ($L_{\text{max}} = 150$):
\begin{equation}
L^* = \text{round} \left( \frac{869.87}{150} \right) \approx 6
\end{equation}
This represents a conservative-efficient configuration. Despite the high geometric stability, the model recognizes that Bitcoin's intrinsic volatility (evident in the large amplitude swings) requires a margin of safety. A $6\times$ leverage allows a trader to capture the structural trend of the attractor while minimizing liquidation risk against stochastic fluctuations that do not break the underlying topology but could otherwise deplete the margin.

\section{Conclusion}

This topological approach provides a rigorous alternative to frequentist risk metrics. By embedding the returns in a high-dimensional space and analyzing the persistence of cycles, we identify the "shape" of the risk. From a mathematical standpoint, this method is superior as it is invariant under coordinate transformations and robust against the heavy-tailed noise characteristic of cryptocurrency markets. The empirical results confirm that BTC maintains a topological coherence that can be proactively exploited to calibrate leverage, treating market risk not just as statistical variance, but as the geometric stability of its phase space.

\end{document}